\documentclass[11pt,a4paper]{article}
\usepackage{amssymb} 
\usepackage{amsmath} 
\usepackage{amsthm} 
\usepackage{hyperref}
\usepackage{enumitem}

\usepackage[percent]{overpic}
\usepackage{todonotes}

\newtheorem{theorem}{Theorem}

\newtheorem{lemma}[theorem]{Lemma}

\newtheorem{conjecture}[theorem]{Conjecture}

\newtheorem{claim}[theorem]{Claim}

\DeclareMathOperator{\bip}{bip}
\DeclareMathOperator{\Obs}{Obs}
\DeclareMathOperator{\Int}{in}
\DeclareMathOperator{\Ext}{ex}

\newcommand{\DeltaH}{{\Delta_H}}
\newcommand{\DeltaG}{{\Delta_G}}
\newcommand{\sigmaG}{{\sigma_G}}

\newcommand*{\myproofname}{Proof}
\newenvironment{claimproof}[1][\myproofname]{\begin{proof}[#1]}{\end{proof}}

\title{Strong cliques and forbidden cycles}

\author{
Wouter Cames van Batenburg
\thanks{Department of Computer Science, Universit\'e Libre de Bruxelles, Belgium. 
Email: \protect\href{mailto:wcamesva@ulb.ac.be}{\protect\nolinkurl{wcamesva@ulb.ac.be}}. Supported by the Netherlands Organisation for Scientific Research (NWO), grant no.~613.001.217, and an ARC grant from the Wallonia-Brussels Federation of Belgium.}
\and
Ross J. Kang
\thanks{Department of Mathematics, Radboud University Nijmegen, Netherlands. 
Email: \protect\href{mailto:ross.kang@gmail.com}{\protect\nolinkurl{ross.kang@gmail.com}}. Supported by a Vidi grant (639.032.614) of the Netherlands Organisation for Scientific Research (NWO).}
\and
Fran\c{c}ois Pirot
\thanks{Department of Mathematics, Radboud University Nijmegen, Netherlands
and LORIA, Universit\'e de Lorraine, Nancy, France.
Email: \protect\href{mailto:francois.pirot@loria.fr}{\protect\nolinkurl{francois.pirot@loria.fr}}.}
}

\begin{document}

\maketitle

\begin{abstract}
Given a graph $G$, the {\em strong clique number} $\omega_2'(G)$ of $G$ is the cardinality of a largest collection of edges every pair of which are incident or connected by an edge in $G$.
We study the strong clique number of graphs missing some set of cycle lengths.
For a graph $G$ of large enough maximum degree $\Delta$, we show among other results the following:
$\omega_2'(G)\le5\Delta^2/4$ if $G$ is triangle-free;
$\omega_2'(G)\le3(\Delta-1)$ if $G$ is $C_4$-free; 
$\omega_2'(G)\le\Delta^2$ if $G$ is $C_{2k+1}$-free for some $k\ge 2$.
These bounds are attained by natural extremal examples. Our work extends and improves upon previous work of Faudree, Gy\'arf\'as, Schelp and Tuza (1990), Mahdian (2000) and Faron and Postle (2019).
We are motivated by the corresponding problems for the strong chromatic index.

\end{abstract}

\section{Introduction}\label{sec:intro}

A {\em strong edge-colouring} of a graph $G$ is a partition of the edges $E(G)$ into parts each of which induces a matching; the {\em strong chromatic index} $\chi_2'(G)$ of $G$ is the least number of parts in such a partition. Although easily defined, $\chi_2'$ has proven very difficult to analyse: a conjecture of Erd\H{o}s and Ne\v{s}et\v{r}il~\cite{Erd88} from the 1980s is notorious. ($\DeltaG$ denotes the maximum degree of $G$.)

\begin{conjecture}[\cite{Erd88}]\label{conj:ErNe}
For a graph $G$ with $\DeltaG=\Delta$, $\chi_2'(G) \le \frac54\Delta^2$.
\end{conjecture}
\noindent
If true, this bound would be sharp for even $\Delta$, by considering a $5$-cycle and substituting each of its vertices with a copy of a stable set of size $\frac12\Delta$.
Molloy and Reed~\cite{MoRe97} proved the existence of some absolute constant $\varepsilon>0$ such that $\chi_2'(G) \le (2-\varepsilon)\Delta^2$ for any graph $G$ with $\DeltaG=\Delta$. Despite significant efforts, cf.~e.g.~\cite{BrJo18}, there is no explicit lower bound on $\varepsilon$ greater than $0.001$ (though better results hold when $\Delta$ is large).

The problem remains difficult if we restrict to a class defined by some forbidden set of subgraphs, say, cycles. The following conjecture of Faudree, Gy\'arf\'as, Schelp and Tuza~\cite{FGST89} has also remained wide open since the 1980s. 

\begin{conjecture}[\cite{FGST89}]\label{conj:FGST}
For a bipartite graph $G$ with $\DeltaG=\Delta$, $\chi_2'(G) \le \Delta^2$.
\end{conjecture}

\noindent
Balanced complete bipartite graphs meet the conjectured bound.
A more general assertion due to Mahdian~\cite{Mah00} is also open: perhaps it suffices to forbid just one in particular of the odd cycle lengths rather than all of them.
\begin{conjecture}[\cite{Mah00}]\label{conj:Mah}
For a $C_5$-free graph $G$ with $\DeltaG=\Delta$, $\chi_2'(G) \le \Delta^2$.
\end{conjecture}

\noindent
By contrast, the best understood classes are when the forbidden set includes some bipartite graph. 
In particular, Mahdian~\cite{Mah00} showed the following for $C_4$. This was subsequently extended to a similar statement with $C_4$ replaced by any bipartite $H$ by Vu~\cite{Vu02}. See~\cite{KaPi18} for an explicit derivation with $H=C_{2k}$.
\begin{theorem}[\cite{Mah00}]\label{thm:Mah}
Fix $\varepsilon >0$. 
If $G$ is a $C_4$-free graph with $\DeltaG=\Delta$, then, provided $\Delta$ is large enough, $\chi_2'(G) \le (2+\varepsilon)\Delta^2/\log \Delta$.
This is sharp up to the multiplicative constant factor.
\end{theorem}

\noindent
Note that graphs of small maximum degree remain a hypothetical obstruction to the outright confirmation of Conjectures~\ref{conj:ErNe} or~\ref{conj:FGST} for $C_4$-free graphs.


Based on the above frustrating state of affairs, it is justified to pursue a simpler parameter than $\chi'_2$, even for restricted graph classes. In particular, a {\em strong clique} of $G$ is a set of edges every pair of which are incident or connected by an edge in $G$;
the {\em strong clique number} $\omega_2'(G)$ of $G$ is the size of a largest such set. No two edges of a strong clique may have the same colour in a strong edge-colouring, so $\omega_2'(G)\le \chi_2'(G)$.
Thus the following classic result is viewed as supporting evidence towards Conjecture~\ref{conj:FGST}.
\begin{theorem}[\cite{FSGT90}]\label{thm:FSGT}
For a bipartite graph $G$ with $\DeltaG=\Delta$, $\omega_2'(G) \le \Delta^2$.
\end{theorem}

By forbidding just one fixed odd cycle length rather than all of them, we refine and improve upon Theorem~\ref{thm:FSGT} as follows.

\begin{theorem}\label{thm:main,odd}
Let $G$ be a graph with $\DeltaG=\Delta$.
\begin{enumerate}
\item\label{itm:triangle} $\omega_2'(G)\le\frac54\Delta^2$ if $G$ is triangle-free.
\item\label{itm:C5} $\omega_2'(G)\le\Delta^2$ if $G$ is $C_5$-free.
\item\label{itm:odd} $\omega_2'(G)\le\Delta^2$ if $G$ is $C_{2k+1}$-free, $k\ge 3$, provided $\Delta \ge 3k^2+10k$. 
\end{enumerate}
\end{theorem}

\noindent
The blown-up $5$-cycles are triangle-free so Theorem~\ref{thm:main,odd}\ref{itm:triangle} is best possible when $\Delta$ is even.
Parts~\ref{itm:C5} and~\ref{itm:odd} of Theorem~\ref{thm:main,odd} are sharp for the balanced complete bipartite graphs. They are a common strengthening and generalisation of Theorem~\ref{thm:FSGT} and a result of Mahdian~\cite[Thm.~15]{Mah00}. Part~\ref{itm:C5} may be viewed as support for Conjecture~\ref{conj:Mah}.
Parts~\ref{itm:triangle} and~\ref{itm:C5} have a common ingredient, but curiously the three proofs for Theorem~\ref{thm:main,odd} are all quite distinct.

We naturally find it interesting to also consider forbidden cycle lengths that are even. 
In this case, we propose the following behaviour.

\begin{conjecture}\label{conj:even}
For a $C_{2k}$-free graph $G$ with $\DeltaG=\Delta$, $\omega_2'(G) \le (2k-1)(\Delta-k+1)$.
\end{conjecture}

\noindent
If true, this is sharp (for $\Delta\ge 2k-2$) by considering some clique on $2k-1$ vertices, to each vertex of which is attached $\Delta-2k+2$ pendant edges. We refer to this construction as a {\em hairy} clique of order $2k-1$.

In support of Conjecture~\ref{conj:even}, we have the following three bounds. The first essentially settles the $k=2$ case. The second is too large by an $O(k)$ factor. The third is almost the conjectured bound, but with the exclusion of two more cycle lengths (also absent in the hairy clique of order $2k-1$).

\begin{theorem}\label{thm:main,even}
Let $G$ be a graph with $\DeltaG=\Delta$.
\begin{enumerate}
\item\label{itm:C4} $\omega_2'(G)\le3(\Delta-1)$ if $G$ is $C_4$-free, provided $\Delta \ge 4$.
\item\label{itm:even} $\omega_2'(G)\le 10k^2(\Delta-1)$ if $G$ is $C_{2k}$-free, $k\ge 3$.
\item\label{itm:3cycles,even} $\omega_2'(G) \le (2k-1)(\Delta-1)+2$ if $G$ is $\{C_{2k},C_{2k+1},C_{2k+2}\}$-free, $k\ge2$.
\end{enumerate}
\end{theorem}

We invite the reader to notice the qualitative difference between excluding an odd cycle length versus an even one. In the latter case, Theorem~\ref{thm:main,even}\ref{itm:even} and Theorem~\ref{thm:Mah} combine to reveal an asymptotic difference in extremal behaviour between the strong clique number and the strong chromatic index. In the former case, it is conjectured that there is no such difference.

We have delved a little further by considering the effect of forbidden (even) cycles within the class of bipartite graphs. For this specific case, we propose the following.

\begin{conjecture}\label{conj:bip}
For a $C_{2k}$-free bipartite graph $G$ with $\DeltaG=\Delta$, $\omega_2'(G) \le k(\Delta-1)+1$.
\end{conjecture}

\noindent
If true, this is sharp (for $\Delta\ge k-1$) by considering a complete bipartite graph $K_{k-1,\Delta}$ with parts of size $k-1$ and $\Delta$, to one vertex in the part of size $\Delta$ is attached $\Delta-k+1$ pendant edges.
In support of Conjecture~\ref{conj:bip}, we have the following result.

\begin{theorem}\label{thm:bip}
For a $\{C_3,C_5,C_{2k},C_{2k+2}\}$-free graph $G$ with $\DeltaG=\Delta$, $\omega_2'(G) \le \max\{k\Delta, 2k(k-1)\}$.
\end{theorem}

\noindent
This is nearly sharp by the example mentioned just above.
A small step in the proof is a curious property of maximising the strong clique number: if we are interested in a class of graphs that are $\{C_3,C_5\}$-free, then we may as well exclude all other odd cycle lengths at the same time (Lemma~\ref{lem:reduction}). This same reduction easily implies a result intermediary to Theorems~\ref{thm:FSGT} and~\ref{thm:main,odd}\ref{itm:C5}.

Let us remark that in general (i.e.~without a cycle restriction) the bound $\omega_2'(G) \le \frac54\Delta^2$ remains conjectural. \'Sleszy\'nska-Nowak~\cite{Sle16} showed a bound of $\frac32\Delta^2$, which was improved to $\frac43\Delta^2$ by Faron and Postle~\cite{FaPo19+}.


\subsection*{Notational conventions}

Let $G$ be a graph or multigraph.
$V(G)$ denotes its vertex set.
$E(G)$ denotes its edge set.
$L(G)$ denotes its line graph.
We write $|G|$ for $|V(G)|$ and $e(G)$ for $|E(G)|=|L(G)|$.

Given a subset $S\subseteq V(G)$, the sub(multi)graph induced by $S$ is denoted by $G[S]$.
Given disjoint subsets $S_1,S_2\subseteq V(G)$, the bipartite sub(multi)graph induced by the edges between $S_1$ and $S_2$ is denoted by $G[S_1\times S_2]$.
We write $E_G(S_2, S_2)$ for $E(G[S_1\times S_2])$ and $e_G(S_2, S_2)$ for $e(G[S_1\times S_2])$.
We write $I_G(S)$  for $E(G[S]) \cup E(G[S\times (V(G)\setminus S)])$, the set of edges incident to $S$.
The neighbourhood $N_G(S)$ of $S$ is the set $\{v' \mid v\in S, v'\notin S, vv'\in E(G)\}$ of neighbours of $S$.
For a vertex $v\in V(G)$, we write $N_G(v)$ instead of $N_G(\{v\})$.
The closed neighbourhood $N_G[v]$ of $v$ is $\{v\}\cup N_G(v)$.
The degree $\deg_G(v)$ of $v$ is $|I_G(\{v\})|$ (and so $\deg_G(v)=|N_G(v)|$ if $G$ is a simple graph).
When the context is clear, we sometimes drop the subscript $G$.

The distance between two vertices is the length of (i.e.~the number of edges in) a shortest path in $G$ that joins them. The distance between two edges is their distance in $L(G)$. The distance between a vertex and an edge is the smaller of the distances between the vertex and the endpoints.
The square $G^2$ of $G$ is formed from $G$ by adding an edge for every pair of vertices that are at distance $2$.
Note $\omega_2'(G)=\omega(L(G)^2)$ and $\chi_2'(G)=\chi(L(G)^2)$ where $\omega$ denotes the clique number and $\chi$ the chromatic number.

\section{No triangles or no $5$-cycles}\label{sec:triangle}

This section is devoted to showing parts~\ref{itm:triangle} and~\ref{itm:C5} of Theorem~\ref{thm:main,odd}.
A common element of the proof is a lemma about the {\em Ore-degree} $\sigmaG$ of $G$, the largest over all edges of $G$ of the sum of the two endpoint degrees.
The following generalises a recent result due to Faron and Postle~\cite{FaPo19+}.

\begin{theorem}\label{thm:C5}
For a $C_5$-free graph $G$, $\omega_2'(G) \le \frac14\sigmaG^2$.
\end{theorem}
\noindent
This directly implies Theorem~\ref{thm:main,odd}\ref{itm:C5} (and hence Theorem~\ref{thm:FSGT}) because the sum of two degrees in $G$ is always at most $2\DeltaG$.
Theorem~\ref{thm:C5} follows from a slightly more technical version.
For a sub(multi)graph $H$ of a (multi)graph $G$, the Ore-degree $\sigma_G(H)$ of $H$ in $G$ is $\max_{xy \in E(H)} (\deg_G(x) + \deg_G(y))$.

\begin{lemma}
\label{lem:C5}
If $G$ is a $C_5$-free multigraph and $H$ is a submultigraph of $G$ such that $E(H)$ is a clique in $L(G)^2$, then $e(H) \le \DeltaH (\sigma_G(H)- \DeltaH) \le  \frac14\sigma_G(H)^2.$
\end{lemma}

\noindent
Before proving this, we first show how it yields Theorem~\ref{thm:main,odd}\ref{itm:triangle}. (In fact, we only need the weaker bipartite version due to Faron and Postle~\cite{FaPo19+}.)

\begin{proof}[Proof of Theorem~\ref{thm:main,odd}\ref{itm:triangle}]
Let $G$ be a triangle-free graph with $\DeltaG=\Delta$.
Let $H$ be a vertex-minimal subgraph of $G$ whose edges form a maximum clique in $L(G)^2$. 
Let $v \in V(G)$ be a vertex satisfying $\deg_H(v)=\DeltaH$.
From now on we call $H$ and its edges \textit{blue}.
Let $G_T = G[V_T]$ and $H_T=H[V_T]$, where $V_T = V(H)\setminus N_G[v]$. 

Let $C_1,C_2,\ldots$ denote the connected components of $H_T$ that contain at least one edge. Fix one such component $C_i$ and let $pq$ be an edge in $C_i$. For all $x \in N_H(v)$, the blue edges $xv$ and $pq$ must be within distance $2$. They are not incident, so either $xp \in E(G)$ or $xq \in E(G)$, but we cannot have both since $G$ is triangle-free. It follows that $N_H(v)$ can be partitioned into $A_i= N_G(p) \cap N_H(v)$ and $\overline{A_i}= N_G(q) \cap N_H(v)$. We will call $\left\{A_i, \overline{A_i} \right\}$ the partition induced by $pq$. Now suppose $C_i$ contains an edge $e$ which is incident to $pq$. Then since $G$ is triangle-free, $e$ and $pq$ must induce the same partition. It follows inductively that all edges in $C_i$ induce the \textit{same} partition $\left\{A_i,\overline{A_i}\right\}$ of $N_H(v)$.
Figure~\ref{fig:trianglefree} illustrates this structure.

\begin{figure}
 \begin{center}
   \begin{overpic}[width=0.55\textwidth]{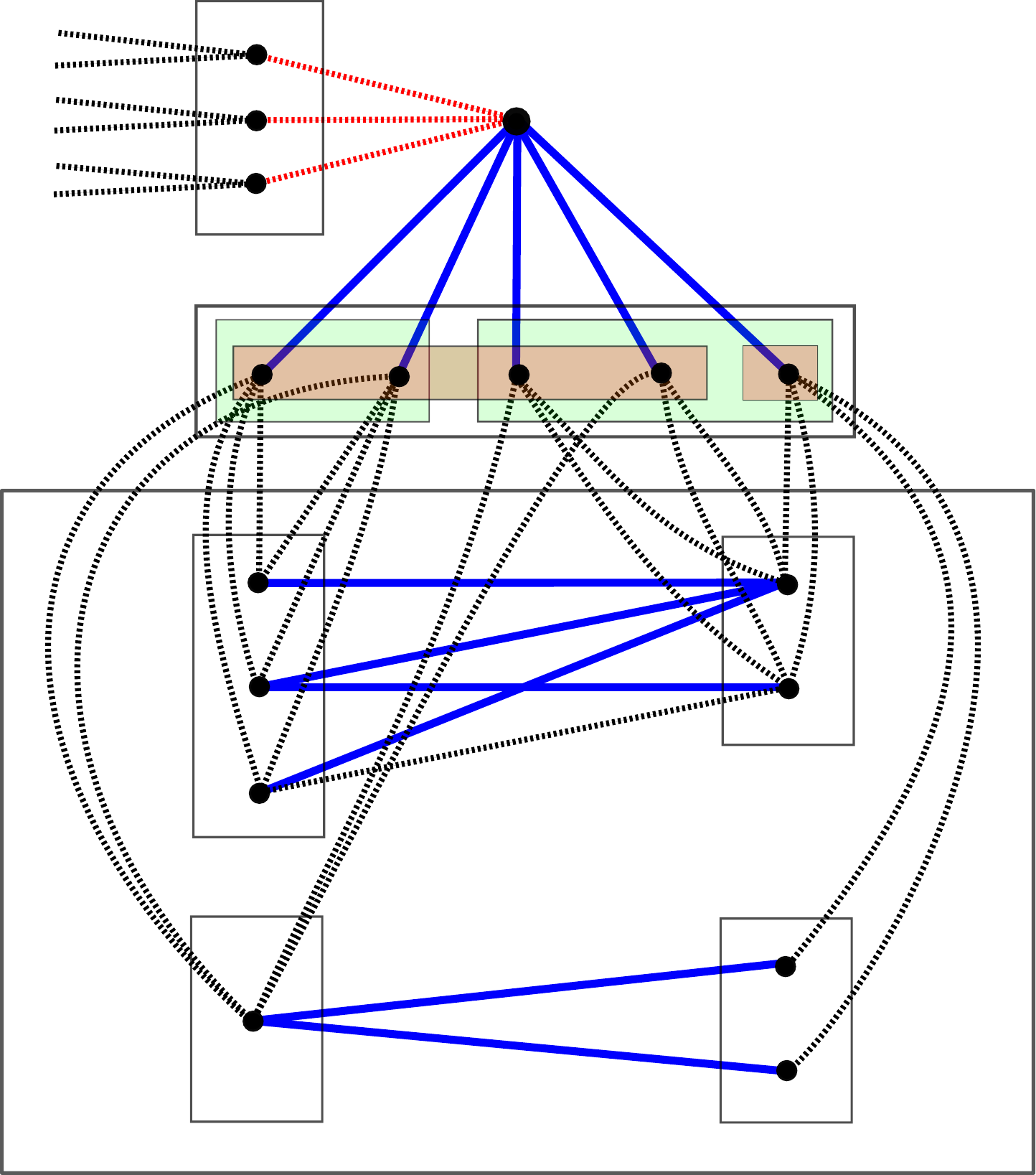}
\put (43,92){$v$}
\put (1,68){$N_H(v)$}
\put (-9,30){$G_T$}
\put (12,102){$N_{G\setminus H}(v)$}
\put (9,44){$X_1$}
\put (74,44){$Y_1$}
\put (9,12){$X_2$}
\put (74,12){$Y_2$}
 \end{overpic}

  \caption{ The structure described in the proof of Theorem~\ref{thm:main,odd}\ref{itm:triangle}. Blue edges are in $H$, red edges are in $G$ but not in $H$, and black edges could be either. In this picture, $H_{\bip}\subseteq G_T$ has two (blue) connected components, induced by $X_1 \cup Y_1$ respectively $X_2 \cup Y_2$. The blue neighbourhood $N_H(v)$ is partitioned into two sets $A_1$ (its left two vertices) and $\overline{A_1}$ (the remaining three vertices on the right), such that $X_1$ is complete to $A_1$ and $Y_1$ is complete to $\overline{A_1}$. The neighbourhoods of $X_2$ and $Y_2$ induce another partition of $N_H(v)$. \textit{Not all} edges are depicted here. In particular, we have left out the (possibly red) edges inside $G_T$ that ensure that all of its blue edges are within distance $2$.
   \label{fig:trianglefree}}
\end{center}
\end{figure}
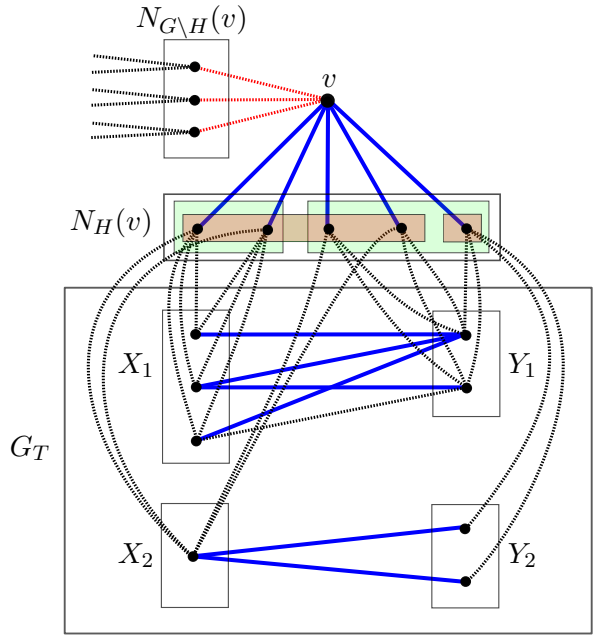

Let $C_1,\ldots, C_k$ be the components (if they exist) that induce the trivial partition $\left\{\emptyset, N_H(v) \right\} $. Let $M=|C_1|+ \cdots +|C_k|$ denote the number of edges that are in these `trivial' components. On the other hand, let $G_{\bip} = G[\bigcup_{i\ge k+1} V(C_i)]$ and $H_{\bip} = H[\bigcup_{i\ge k+1} V(C_i)]$ be the graphs induced by the remaining `nontrivial' components.

\begin{claim}\label{clm:1}
$M \le (\Delta-\DeltaH) \Delta$.
\end{claim}

\begin{claim}\label{clm:2}
$\sigma_{G_{\bip}}(H_{\bip}) \le 2\Delta-\DeltaH -M/\Delta$.
\end{claim}

\begin{claim}\label{clm:3}
$G_{\bip}$ is bipartite.
\end{claim}

Before proving these claims, we show how they imply the theorem. Note that $E(H_{\bip})$ is a clique not only in $L(G)^2$ but also in $L(G_{\bip})^2$. So by Claim~\ref{clm:3}, we may apply Lemma~\ref{lem:C5} and then Claim~\ref{clm:2}, yielding
\[
e(H_{\bip}) \le \frac14\sigma_{G_{\bip}}(H_{\bip})^2 \le \frac14\left(2\Delta-\DeltaH- \frac{M}{\Delta}\right)^2.
\]
It follows that $\omega(L(G)^2)$ is at most
\begin{align*}
e(H)
&= |I_H(N_G(v))| + e(H_T)
 \le \DeltaH \deg_G(v) + M + e(H_{\bip}) \\
& \le \DeltaH\Delta + M + \frac14\left(2\Delta-\DeltaH- \frac{M}{\Delta}\right)^2 \\
&
= \Delta^2 + \frac{1}{4} \left(\DeltaH+ \frac{M}{\Delta} \right)^2 \le \frac{5}{4} \Delta^2,
\end{align*}
where we used Claim~\ref{clm:1} in the last line.
This concludes the proof, conditioned on Claims~\ref{clm:1}--\ref{clm:3}.

Given the $i$th component $C_i$, let $X_i$ respectively $Y_i$ denote the set of vertices in $C_i$ whose neighbourhood in $N_H(v)$ is $A_i$ respectively $\overline{A_i}$. Note that $X_i$ is complete to $A_i$ and $Y_i$ is complete to $\overline{A_i}$. Furthermore, the bipartite subgraph of $H$ induced by $C_i$ has parts $X_i$ and $Y_i$.

\begin{claimproof}[Proof of Claim~\ref{clm:1}]
If $C_i$ is a trivial component $(1\le i \le k$) then $Y_i$ is complete to $\overline{A_i}=N_H(v)$. Therefore $|\bigcup_{1\le i \le k} Y_i| \le \Delta$, and for the same reason all $y\in \bigcup_{1\le i \le k} Y_i$ satisfy $\deg_{H_T}(y)\le \Delta - \DeltaH$. So $M\le  \sum_{y\in \bigcup_{1\le i \le k} Y_i}\deg_{H_T}(y) \le \Delta (\Delta-\DeltaH)$.
\end{claimproof}

\begin{claimproof}[Proof of Claim~\ref{clm:2}]
Let $e=pq \in E(H_{\bip})$. Then for all $x \in N_H(v)$, $x$ must be adjacent to either $p$ or $q$. So there are $\deg_H(v)=\DeltaH$ edges between $\{p,q\}$ and $N_H(v)$. Also, $pq$ must be at distance $2$ of each of the $M$ edges induced by the trivial components. So there are at least $M/\Delta$ edges between $\{p,q\}$ and the trivial components. So at least $\DeltaH+ M/\Delta$ edges incident to $\{p,q\}$ are not in $G_{\bip}$. It follows that $\sigma_{G_{\bip}}(e)= \deg_{G_{\bip}}(p) + \deg_{G_{\bip}}(q) \le 2\Delta- \DeltaH - M/\Delta$. 
\end{claimproof}

\begin{claimproof}[Proof of Claim~\ref{clm:3}]
Suppose there are two different nontrivial components, $C_i$ and $C_j$. We will first show that we may then assume that either $A_i \subseteq A_j$ or $A_j \subseteq A_i$. Indeed,
if either $\overline{A_j} \subseteq A_i$ or $A_i \subseteq \overline{A_j}$, then after interchanging $X_j$ and $Y_j$ (and thus interchanging $A_j$ and $\overline{A_j}$), we obtain $A_j \subseteq A_i$ or $A_i \subseteq A_j$. So we may assume for a contradiction that none of $A_i\subseteq A_j, A_i \subseteq \overline{A_j}, A_j \subseteq A_i, \overline{A_j} \subseteq A_i$ holds. But then there exist $a \in A_i \cap \overline{A_j}, b \in A_i \cap A_j, c \in \overline{A_i} \cap A_j$ and $d \in \overline{A_i} \cap \overline{A_j}$. Furthermore, because each component contains at least one blue edge, there are blue edges $(x_i, y_i) \in X_i \times Y_i$ and $(x_j,y_j) \in X_j \times Y_j$ that have to be connected by an edge in order to have them within distance $2$. If $x_ix_j$ is an edge, then $x_ix_jb$ forms a triangle. Similarly, if $x_iy_j$, $y_iy_j$ or $x_jy_i$ is an edge then $x_iy_ja$, $y_iy_jd$ or $x_jy_ic $ is a triangle, respectively. Contradiction.

It follows that we can reorder the components by inclusion, so that $A_{k+1} \subseteq A_{k+2} \subseteq \cdots$. Now we are ready to show that $G_{\bip}$ is bipartite on parts $X= \bigcup_{i\ge k+1} X_i$ and $Y= \bigcup_{i\ge k+1} Y_i$. Suppose $X$ is not a stable set. Then there are $x_1,x_2 \in X$ that form an edge, where $x_1 \in X_i$ and $x_2 \in X_j$ for some $i\le j$. Since $\emptyset \ne A_i \subseteq A_j$, there must be a triangle in $x_1x_2A_i$. Contradiction. Similarly, suppose $Y$ is not a stable set. Then there are $y_1,y_2 \in Y$ that form an edge, where $y_1 \in Y_i$ and $y_2 \in Y_j$ for some $i\le j$. Since $\emptyset \ne\overline{A_j} \subseteq \overline{A_i}$, there must be a triangle in $y_1y_2\overline{A_j}$. Contradiction.
\end{claimproof}
This completes the proof of Theorem~\ref{thm:main,odd}\ref{itm:triangle}.
\end{proof}

\begin{proof}[Proof of Lemma~\ref{lem:C5}]
Let $G$ be a $C_5$-free multigraph and $H$ be a submultigraph of $G$ whose edges form a maximum clique in $L(G)^2$.  Let $v \in V(G)$ be a vertex satisfying $\deg_H(v)=\DeltaH$.

We may assume that $|N_H(v)| \ge 2$. Indeed if $|N_H(v)|=1$ then, writing $N_H(v)=\{u\}$, the multiplicity of $uv$ in $H$ is equal to $\DeltaH$. Each vertex in $N_G(\{u,v\})$ is incident to at most $\DeltaH$ edges of $H$, and there are at most $\deg_G(u)+\deg_G(v)-2\DeltaH\le\sigma_G(H)-2\DeltaH$ such vertices. Therefore  $e(H) \le \DeltaH + (\sigma_G(H)-2\DeltaH) \DeltaH\le \DeltaH(\sigma_G(H)- \DeltaH)$, as desired.

Now let $E^* \subseteq E(H)$ denote the set of those edges $st \in E(H)$ for which $s, t \notin N_G(v)$.  Let $st \in E^*$. Then, for all $u \in N_H(v)$, $vu$ must be within distance $2$ of $st$, so either $us \in E(G)$ or $ut \in E(G)$. Without loss of generality, $us \in E(G)$. Because $G$ has no $C_5$ and $|N_H(v)| \ge 2$, it follows that $t$ is anticomplete to $N_H(v)\setminus\{u\}$, so in fact $s$ is complete to $N_H(v)$ and $t$ is anticomplete to $N_H(v)$. We derived this for all $st \in E^*$, so there exists a subset $S \in V(H)$ such that 
\begin{enumerate}
\item \label{bla} $E^* \subseteq I_H(S)$, and
\item \label{blabla} $S$ is complete to $N_H(v)$. 
\end{enumerate}

Since each edge of $H$ is either in $E^*$ (and thus has an endpoint in $S$) or has an endpoint in $N_G(v)$, we can cover $E(H)$ with the following subsets:
\begin{align*}
E_S&=I_H(S), \,\,\,\,\,
E_1=I_H(N_G(v)\setminus N_H(v)), \text{ and }\,
E_2=I_H(N_H(v))\setminus E_S.
\end{align*}
 Each vertex is incident to at most $\DeltaH$ edges of $H$, so $|E_S|\le  \DeltaH |S|$. Furthermore, $|E_1|\le \DeltaH |N_G(v)\setminus N_H(v)|  \le \DeltaH (\deg_G(v) - \DeltaH)$ by our choice of $v$. By property~\ref{blabla}, each vertex  $x \in N_H(v)$ is incident to at most $\deg_G(x)-|S|$ edges that are not incident to $S$. Thus, $|E_2|\le \sum_{x \in N_H(v)} (\deg_G(x) - |S|) \le \DeltaH (\sigma_G(H) - \deg_G(v)- |S|)$. In conclusion,

\[
e(H) \le |E_S|+|E_1|+|E_2| \le \DeltaH (\sigma_G(H)-\DeltaH).
\]
The last expression is largest if $\DeltaH= \frac12\sigma_G(H)$, so $e(H)\le \frac14\sigma_G(H)^2$.
\end{proof}

\section{No $(2k+1)$-cycles or no $2k$-cycles}\label{sec:odd}\label{sec:even}

In this section, we prove Theorem~\ref{thm:main,odd}\ref{itm:odd} and Theorem~\ref{thm:main,even}\ref{itm:even}. The methods are quite different from those of the previous section. 
In both proofs, we utilise a Tur\'an-type lemma for graphs with no path $P_{2k+1}$ of order $2k+1$.

Given a graph $G$ and a subset $S\subseteq V(G)$, an edge $uv$ out of $S$, with $u\in S$ and $v\notin S$, say, is called {\em $b$-branching out of $S$} if $|N_G(v) \cap S| \ge b$.

\begin{lemma}\label{ksat}
Fix $k\ge 2$.
Given a graph $G$ and a subset $X\subseteq V(G)$, if $G[X\times (V(G)\setminus X)]$ is $P_{2k+1}$-free, then
at most $k^2 |X|$ edges are $k$-branching out of $X$.
\end{lemma}

\begin{proof}
Let $G$ and $X\subseteq V(G)$ satisfy the hypothesis.
Without loss of generality, we may assume that $G$ is connected.
We prove the result by induction on $|X|$.
The statement is trivially true if $|X| < k$.
For the induction, let us assume that $|X| \ge k(\ge 2)$ and that the statement is true for any $X'\subseteq V(G)$ with $|X'|<|X|$.

Suppose there is some $x\in X$ that is incident to at most $k$ edges that are $k$-branching. We then have by induction that there are at most $k^2(|X|-1)$ edges that are $k$-branching out of $X\setminus\{x\}$. 
Let us call a vertex $v \notin X$ \textit{pivotal} if there exists an edge incident to $v$ which is $k$-branching out of $X$ but not $k$-branching out of $X\setminus\{x\}$. Each pivotal vertex must be adjacent to $x$ and must have exactly $k-1$ neighbours in $X\setminus\{x\}$. It follows that there are at most $k$ pivotal vertices, each of which is incident to exactly $k$ edges that are $k$-branching out of $X$ (but not $k$-branching out of $X\setminus\{x\}$). In conclusion, at most $k^2(|X|-1) + k^2=k^2|X|$ edges are $k$-branching out of $X$.

There remains the possibility that every vertex of $X$ is incident to  at least $k+1$ edges that are $k$-branching.
In this case, however, we can construct a path through the following iterative process. Consider an arbitrary $k$-branching edge out of $X$ and let $x'_0$ be its endpoint in $V(G)\backslash X$. Suppose we have constructed a path $x'_0x_0\cdots x'_{i-1}x_{i-1}x'_i$ for some $i\ge 0$, such that $x'_i$ is incident to an edge which is $k$-branching out of $X$. Note that $x'_i$ has at least $k$ neighbours in $X$ by definition. If $i < k$, we may choose $x_{i}\in X\setminus\{x_0,\dots,x_{i-1}\}$. Subsequently, we may choose a $k$-branching edge $x_ix'_{i+1}$ such that $x'_{i+1} \in V(G)\setminus (X\cup\{x'_0,\dots,x'_{i}\})$. This process certifies that $x'_0x_0\cdots x'_k$ is in $G[X\times (V(G)\setminus X)]$. This path is of order $2k+1$, contrary to our assumption.
\end{proof}

In addition to Lemma~\ref{ksat}, we need a basic bound on the Tur\'an number of a path $P_{\ell+1}$ of order $\ell +1$ due to Erd\H{o}s and Gallai~\cite{ErGa59}.

\begin{lemma}[\cite{ErGa59}]\label{lem:ErGa}
For a $P_{\ell+1}$-free graph $G$, $e(G)\le (\ell-1)|G|/2$.
\end{lemma}

\begin{proof}[Proof of Theorem~\ref{thm:main,odd}\ref{itm:odd}]
Let $G$ be a $C_{2k+1}$-free graph with $\DeltaG=\Delta\ge 16k$.
Let $H$ be a subgraph of $G$ whose edges form a maximum clique in $L(G)^2$.
Thus bounding $\omega(L(G)^2)$ is equivalent to bounding $e(H)$.
Let $v \in V(G)$ be a vertex satisfying $\deg_H(v)=\DeltaH$. 
For short, we write $A_v = N_G(v)$ and $B_v=N_G(N_G[v])$.
Note that if $E(H[B_v])=\emptyset$,
then every edge in $H$ is incident to $A_v$, and $e(H) \le \deg_G(v)\Delta \le \Delta^2$. 

For any edge $xy\in E(H[B_v])$, it must hold that $N_H(v) \subseteq N_G(\{x,y\})$. If, for every such edge, it holds that $N_G(x)\cap N_H(v) = \emptyset$, say, (and so $N_H(v) \subseteq N_G(y)$,) then $e(H) \le \deg_G(v)\DeltaH + \Delta(\Delta-\DeltaH)\le \Delta^2$.

Fix $xy\in E(H[B_v])$ such that $N_G(x)\cap N_H(v) \ne \emptyset$ and $N_G(y)\cap N_H(v) \ne \emptyset$.
Let $\left\{X,Y\right\}$ be a non-trivial partition of $N_H(v)$ such that $X \subseteq N_G(x)$ and $Y\subseteq N_G(y)$.
We will need the following claim.

\begin{claim}\label{clm:odd}
\begin{enumerate}
\item\label{itm:odd,2k-1} $G[A_v]$ is $P_{2k}$-free. 
\item\label{itm:odd,2k-2} $G[X\times (B_v\setminus\{x,y\})]$ and $G[Y\times (B_v\setminus\{x,y\})]$ are $P_{2k-1}$-free.
\end{enumerate}
\end{claim}

\begin{proof}
\begin{enumerate}[wide, labelwidth=!, labelindent=0pt]
\item If $G[A_v]$ contains a path $P$ of order $2k$, then the concatenation of $P$ with $v$ is a cycle of order $2k+1$ in $G$, a contradiction.
\item If $G[X\times (B_v\setminus\{x,y\})]$ (say) contains a path of order $2k-1$, then there must be a subpath $P$ of order $2k-3$ such that both its endpoints are in $X$. Now for any $b\in Y$, the concatenation of $P$ with the path $vbyx$ is a cycle of length $2k+1$ in $G$, a contradiction.
\qedhere
\end{enumerate}
\end{proof}

Letting $C$ denote the set of $(2k-3)$-branching edges out of $N_H(v)$ with respect to the graph $G[N_H(v)\times (B_v\setminus\{x,y\})]$, note that the edges of $H$ can be covered by the following six sets:
\begin{align*}
&I_H(\{v,x,y\}), \,\,\,\,\,I_H(N_G(v)\setminus N_H(v)), \,\,\,\,\,E(H[A_v]), \\
&C, \,\,\,\,\,E_H(N_H(v),B_v\setminus\{x,y\})\setminus C, \,\,\,\,\,E(H[B_v\setminus\{x,y\}]).
\end{align*}

We have $|I_H(\{v,x,y\})| \le 3\DeltaH-1$ and $|I_H(N_G(v)\setminus N_H(v))| \le (\Delta-\DeltaH)\DeltaH$.
Also, $e(H[A_v]) \le (k-1)\Delta$ by Lemma~\ref{lem:ErGa} and Claim~\ref{clm:odd}\ref{itm:odd,2k-1}.

Let $C_x$ (resp.~$C_y$) be the set of $(k-1)$-branching edges out of $X$ (resp.~$Y$) with respect to the graph $G[X\times (B_v\setminus\{x,y\})]$ (resp.~$G[Y\times (B_v\setminus\{x,y\})]$). 
By Lemma~\ref{ksat} and Claim~\ref{clm:odd}\ref{itm:odd,2k-2}, $|C_x| \le (k-1)^2|X|$ and  $|C_y| \le (k-1)^2|Y|$.
By the pigeonhole principle, for every set of edges of $C$ incident to a single vertex in $B_v\setminus\{x,y\}$, more than half of them belong to $C_x$ or to $C_y$.  Thus $|C| \le 2((k-1)^2|X|+(k-1)^2|Y|)=2(k-1)^2\DeltaH$. 
We next bound $|E_H(N_H(v),B_v\setminus\{x,y\})\setminus C|$.

Supposing that there are edges in $E_H(N_H(v), B_v\setminus\{x,y\})\setminus C$, let $A'$ be the set of vertices in $N_H(v)$ that are incident to such an edge and let $u\in A'$ be a vertex of minimum degree in $H[A']$.
Note that $\deg_{H[A']}(u) \le 2(k-1)$ by Lemma~\ref{lem:ErGa} and Claim~\ref{clm:odd}\ref{itm:odd,2k-1}.
Let $uw\in E_H(N_H(v), B_v\setminus\{x,y\})$ be an edge that is not $(2k-3)$-branching out of $N_H(v)$. 
By crudely bounding the number of edges that are within distance $2$ of $uw$ and not $(2k-3)$-branching, we have that
\begin{align*}
&|E_H(N_H(v), B_v\setminus\{x,y\})\setminus C|\\
&\le |N_G(u)\cap (B_v\setminus\{x,y\})|\cdot(2k-4)
+|N_G(w)\cap (B_v\setminus\{x,y\})|\cdot (2k-4)\\
&\,\,+|N_G(u)\cap A'|\cdot \Delta
+(|N_G(w)\cap A'|-1)\cdot \Delta\\
&\le \Delta (2k-4)
+\Delta (2k-4) +2(k-1)\Delta
+(2k-5)\Delta\\
&=(8k-15)\Delta.
\end{align*}

It only remains to bound the number of edges of $H[B_v\setminus\{x,y\}]$. First observe that we may assume $\deg_H(v)=\DeltaH\ge 4k-6$, for otherwise $e(H) \le |N_H[\{u,w\}]|\DeltaH < 2(4k-6)\Delta$ (which is at most $\Delta^2$ if $\Delta\ge 16k$).  
Let $x'y'$ be an edge of $H[B_v\setminus\{x,y\}]$.
Write $d_{x'}$ for $|N_G(x')\cap N_H(v)|$ and $d_{y'}$ for $|N_G(y')\cap N_H(v)|$. As we have already observed earlier, $N_G(\{x',y'\})\supseteq N_H(v)$, and so $d_{x'}+d_{y'} \ge \DeltaH$.
If both $d_{x'} \ge 2k-3$ and $d_{y'} \ge 2k-3$, then the number of edges of $C$ incident to $\{x',y'\}$ is at least $d_{x'}+d_{y'} \ge \DeltaH$, and so one of $x'$ or $y'$, say $y'$, is incident to at least $\DeltaH/2$ edges of $C$.
Otherwise, $d_{x'} \le 2k-4$ (say) and so $d_{y'}\ge \DeltaH-2k+4 \ge \DeltaH/2$, in which case $y'$ is incident to at least $\DeltaH/2$ edges of $C$.
In either case, the number of edges of $H[B_v\setminus\{x,y\}]$ incident to $y'$ is at most $\Delta-\DeltaH/2$.
Since $x'y'$ was arbitrary, what we have shown is that $H[B_v\setminus\{x,y\}]$ admits a vertex cover each member of which is incident to at least $\DeltaH/2$ edges of $C$ and to at most $\Delta-\DeltaH/2$ edges of $H[B_v\setminus\{x,y\}]$. The size of this vertex cover is at most $|C|$ divided by $\DeltaH /2$. It then follows, using our earlier derived bound on $|C|$, that
\begin{align*}
e(H[B_v\setminus\{x,y\}])
& \le \frac{2(k-1)\DeltaH}{\DeltaH/2} (\Delta-\DeltaH/2) \\
& = 2(k-1)(2\Delta-\DeltaH).
\end{align*}

Summing all of the above estimates, we deduce that the number of edges in $H$ is at most
\begin{align*}
&3\DeltaH-1
+(\Delta-\DeltaH)\DeltaH
+(k-1)\Delta\\
&+2(k-1)^2\DeltaH
+(8k-15)\Delta
+2(k-1)(2\Delta-\DeltaH)\\
&\le \frac14\Delta^2+ (2k^2+7k-13)\Delta,
\end{align*}
which is at most $\Delta^2$ if $\Delta \ge 3k^2 +10k$.
\end{proof}

\begin{proof}[Proof of Theorem~\ref{thm:main,even}\ref{itm:even}]
Let $G$ be a $C_{2k}$-free graph with $\DeltaG=\Delta$.
Let $H$ be a subgraph of $G$ whose edges form a maximum clique in $L(G)^2$.
Choose an edge $uv\in E(H)$, and define
$A_u= N_G(u) \setminus \{v\}$ and
$A_v= N_G(v) \setminus N_G[u]$. For short we will write $A= A_u \cup A_v$ for the neighbourhood of $\{u,v\}$ and we also need the second-order neighbourhood $B= N_G(A)\setminus (A \cup \{u,v\})$.
Note that $|A| \le 2\Delta-2$.

Letting $C$ denote the set of $(2k-1)$-branching edges out of $A\cup \{u,v\}$, note that the edges of $H$ can be covered by the following four sets:
\[
I_H(\{u,v\}), \,\,\,\,\,E(H[A]), \,\,\,\,\,C, \,\,\,\,\,E_H(A,B)\setminus C.
\]

\begin{claim}\label{clm:even}
\begin{enumerate}
\item\label{itm:even,2k-2}
 $G[A]$ is $P_{2k-1}$-free.
\item\label{itm:even,2k}
 $G[A_u\times B]$ and $G[A_v\times B]$ are $P_{2k+1}$-free.
\end{enumerate}
\end{claim}

\begin{claimproof}
\begin{enumerate}[wide, labelwidth=!, labelindent=0pt]
\item Suppose that there is a path $x_0\cdots x_{2k-2}$ of order $2k-1$ in $G[A_u\cup A_v]$.
Without loss of generality assume that $x_0\in A_u$. Then $x_{2k-3}\notin A_v$ or else $ux_0\cdots x_{2k-3}v$ would be a cycle of length $2k$ in $G$, a contradiction. So $x_{2k-3}\in A_u$. Also $x_{2k-2}\in A_v$ or else $ux_0\cdots x_{2k-2}$ would be a cycle of length $2k$. Then $x_1\in A_v$ or else $ux_1\cdots x_{2k-2}v$ would be a cycle of length $2k$.
Now take the least $i$ such that $x_{2i-1}\in A_v$ and $x_{2i+1}\in A_u$. Since $x_1\in A_v$ and $x_{2k-3}\in A_u$, such an $i \in [2k-2]$ exists. But this implies that $ux_0\cdots x_{2i-1}vx_{2k-2}\cdots x_{2i+1}$ is a cycle of length $2k$, a contradiction.
\item In any path of order $2k+1$ in $G[A_u\times B]$, say, there is a path of order $2k-1$ with both of its endpoints in $A_u$. This path together with $u$ forms a $C_{2k}$ in $G$, a contradiction.
\qedhere
\end{enumerate}
\end{claimproof}

Note that $|I_H(\{u,v\})| \le 2\Delta-1$.
By Lemma~\ref{lem:ErGa} and Claim~\ref{clm:even}\ref{itm:even,2k-2}, $e(H[A]) \le (k-1)|A|$. 
Let $C_u$ (resp.~$C_v$) be the set of $k$-branching edges out of $A_u\cup \{u\}$ (resp.~$A_v\cup \{v\}$) with respect to the graph $G[A_u \cup V(G)\setminus (A_v  \cup \{v\})]$ (resp.~$G[A_v \cup V(G)\setminus (A_u  \cup \{u\})]$). 
By Lemma~\ref{ksat} and Claim~\ref{clm:even}\ref{itm:even,2k}, $|C_u| \le k^2|A_u|$ and  $|C_v| \le k^2|A_v|$.
By the pigeonhole principle, for every set of edges of $C$ incident to a single vertex in $B$, more than half of them belong to $C_u$ or to $C_v$.  Thus $|C| \le 2(k^2|A_u|+k^2|A_v|)=2k^2|A|$. 
It remains to bound $|E_H(A,B)\setminus C|$.

Assuming that $E_H(A,B)\setminus C\ne \emptyset$,
let $A'\subseteq A$ be the set of those vertices in $A$ that are incident to some edge of $E_H(A,B)\setminus C$ and let $x\in A'$ be a vertex of minimum degree in $H[A']$. Note that $\deg_{H[A']}(x) \le 2k-3$ by Lemma~\ref{lem:ErGa} and Claim~\ref{clm:even}\ref{itm:even,2k-2}. 
Let $xy\in E(H)$ be an edge that is not $(2k-1)$-branching out of $A\cup\{u,v\}$. 
By crudely bounding the number of edges that are within distance $2$ of $xy$ and not $(2k-1)$-branching, we have that
\begin{align*}
|E_H(A,B)\setminus C|
\le \,& |N_G(x)\cap B| \cdot (2k-2)
+|N_G(y)\cap B|\cdot(2k-2)\\
&+|N_G(x)\cap A'|\cdot \Delta
+(|N_G(y)\cap A'|-1)\cdot \Delta\\
\le \,& \Delta(2k-2)
+\Delta(2k-2)
+(2k-3)\Delta
+(2k-3)\Delta\\
=\,& (8k-10)\Delta.
\end{align*} 

Combined with the previous estimates and using that $\Delta \ge 3$ (otherwise $G$ is a collection of vertex-disjoint trees and cycles, so the theorem follows straightforwardly), we obtain
\begin{align*}
e(H)
&\le 2\Delta-1+(2k^2+k-1)|A|+(8k-10)\Delta\\
&\le (4k^2+10k-10)\Delta +1-2k-2k^2\le (5k^2+14k)(\Delta-1).\qedhere
\end{align*}
\end{proof}

\section{No three consecutive cycle lengths}\label{sec:3cycles}

In this section, we prove a stronger version of Theorem~\ref{thm:main,even}\ref{itm:3cycles,even}.

\begin{theorem}\label{thm:3cycles}
Let $G$ be a graph with $\DeltaG=\Delta$.
\begin{enumerate}
\item\label{itm:path}
$\omega(L(G)^2)\le (\kappa-2) (\Delta-1) + 2$
if $G$ is $P_{\kappa+1}$-free, $\kappa\ge 3$.
\item\label{itm:3cycles}
$\omega(L(G)^2)\le (\ell-2) (\Delta-1) + 2$
if $G$ is $\{C_{\ell-1},C_{\ell},C_{\ell+1}\}$-free, $\ell\ge5$.
\end{enumerate}
\end{theorem}

\begin{proof}[Proof of Theorem~\ref{thm:3cycles}]
If $\Delta=1$ then $\omega(L(G)^2)=1$. If $\Delta =2$, then $G$ is a path or a cycle, or a vertex-disjoint union of such graphs. For all such graphs, $\omega(L(G)^2)\le 5 \le (\ell-2) (\Delta-1)+2$. 
Furthermore, if $G$ is not a $4$-cycle or $5$-cycle then $\omega(L(G)^2)\le 3$, so $\omega(L(G)^2) \le (\kappa-2) (\Delta-1)+2$. Thus we may assume from now on that $\Delta\geq 3$. The idea of the proof is to assume that $\omega(L(G)^2)$ is large and then iteratively construct a path of order $l+1$ (respectively $k+1$) whose extremal edges are in $E(H)$. This will imply the existence of a cycle (respectively path) of forbidden length; contradiction.

 Let $H$ be a subgraph of $G$ whose edges form a maximum clique in $L(G)^2$. Note that $e(H)=\omega(L(G)^2)>\Delta$, for otherwise the conclusion of the theorem is already satisfied. It follows that $G$ contains a path $P_4=x_1y_1x_2y_2$ that starts and ends on edges $x_1y_1, x_2y_2$ from $E(H)$. Indeed, let $e_1$ and $e_2$ be edges of $E(H)$. If they are not incident to each other, then there must be an edge between them and we have obtained the desired $P_4$. So we may assume that all edges of $E(H)$ are pairwise incident and in particular we can write $e_1=xy$ and $e_2=yz$. At most $\Delta$ edges meet in $y$, so $E(H)$ contains an edge $e_3$ that is not incident to $y$ and therefore $e_3$ is incident to $x$. If $e_3=xq \ne xz$ then $qxyz$ forms the desired $P_4$. Otherwise $xyz$ forms a triangle of edges from $E(H)$. Since $e(H)\ge \Delta+1 \ge 4$, there is a fourth edge in $E(H)$ incident to the triangle, again yielding a $P_4$. 

We now define the paths $W_1:=y_1x_2$ and  $W_1^*:=x_1y_1x_2y_2$, the latter being the $P_4$ whose existence we derived above. These paths serve as the initialisation step of a construction (described below). The input of this construction is given by a path $W_i:=y_1x_2y_2\ldots y_{i-1}x_i$ and a `preliminary' path $W_i^*:=x_1W_iy_i$, with the property that the first and final edge of $W_i^*$ are in $E(H)$. The output consists of {\em longer} paths $W_{i+1}$ and $W_{i+1}^*$, with the same properties.
For this construction to work, we need the edge set  $F_i:=E(H) \setminus ( I_H(W_i)\cup \{x_1y_i\} )$ to be nonempty.

As long as $F_i$ is nonempty, we iterate the following case consideration.
\begin{enumerate}[wide, labelwidth=!, labelindent=0pt,noitemsep,font=\bfseries,label=Case \arabic*.,ref=\arabic*]
\item\label{Acase1}
{\em $F_i$ contains an edge which is incident to the first vertex $x_1$ or the last vertex $y_i$ of $W_i^*$.} 

Choose such an edge $e_{i+1} \in F_i$ and assume without loss of generality that it is incident to $y_i$. Then we add $e_{i+1}$ to our preliminary path, and we set $W_{i+1}^*=W_i e_{i+1}$ and $W_{i+1}=W_i^*$. By the definition of $F_i$, $y_i$ is the only vertex in $W_i^*$ that is incident to $e_{i+1}$, so $W_{i+1}^*$ is a path as well.

\item\label{Acase2}
{\em Case~\ref{Acase1} does not apply.}

Then $F_{i}$ contains an edge $x_{i+1}y_{i+1}$ which is not incident to $x_1$ or $y_i$. By the definition of $F_i$, $x_{i+1}y_{i+1}$ is not incident to $x_i$ either.  Therefore there must be an edge $e^*$ between $x_{i+1}y_{i+1}$ and $x_iy_i$. Without loss of generality, $e^*$ is incident to $x_{i+1}$, so we have $x_ix_{i+1} \in E(G)$ or $y_i x_{i+1} \in E(G)$.

\begin{enumerate}[wide, labelwidth=!, labelindent=0pt,noitemsep,font=\bfseries,label=Subcase \arabic{enumi}.\arabic*.,ref=\arabic{enumi}.\arabic*]
\item\label{Acase2.1}
{\em $x_{i}x_{i+1} \in E(G)$.}

Then we set $W_{i+1}^*= W_i x_{i+1} y_{i+1}$ and $W_{i+1}=W_i x_{i+1}$.

\item\label{Acase2.2}
{\em $y_{i}x_{i+1} \in E(G)$ and subcase \ref{Acase2.1} does not apply.}

Then we set $W_{i+1}^*=W_i y_i x_{i+1} y_{i+1}=W_i^*x_{i+1}y_{i+1}$ and $W_{i+1}=W_i y_i x_{i+1}$.
\end{enumerate}
\end{enumerate}

After the final iteration $I$, the set  $F_I$ is empty. 
Since 
\(
F_I= E(H) \setminus ( I_H(W_I)\cup \{x_1y_I\} )
\)
and because the number of edges incident to $W_I$ is at most  $(\Delta-1) |W_I| +1$, it follows that $0=|F_I| \ge e(H) -2 -(\Delta-1) |W_I| $. 

Because $G$ is $P_{\kappa+1}$-free,  our constructed path $W_I^*$ cannot be too large. More precisely, we must have $\kappa \ge |W_I^*| =|W_I|+2$, and therefore 
\[
\omega(L(G)^2) =e(H) \le  (\Delta-1) |W_I| +2 \le (\Delta-1) (\kappa-2) +2.
\]
This concludes the proof of~\ref{itm:path}. For~\ref{itm:3cycles}, we extend the argument slightly.
Suppose for a contradiction that $\omega(L(G)^2)=e(H) \ge (\ell-2) (\Delta-1)+3$.  Then  $W_I^*$ is a path on $|W_I|+2 \ge \tfrac{e(H)-2}{\Delta-1} + 2 \ge \ell + \tfrac{1}{\Delta-1}$ vertices. Note that in the $i$th iteration, the order of the path $W_i^*$ is increased by either $1$ or $2$.\footnote{$|W_{i+1}^*| - |W_i^*|$ is equal to one in case \ref{Acase1} and subcase \ref{Acase2.1}, and equal to two in subcase \ref{Acase2.2}.}   Therefore there exists a $j \le I$ such that $|W_j^*| \in \{\ell,\ell+1\}$.

From now on, let us call the edges of $E(H)$ \textit{blue} and the other edges of $E(G)$  \textit{red}. First we derive that it suffices to show the existence of a $P_{\ell+1}$ that starts and ends on blue edges. Suppose $G$ has a path $A$ of order $\ell+1 \ge 6$ that starts with a blue edge $a_1a_2$ and ends on another blue edge $a_{\ell} a_{\ell+1}$. These (nonincident) blue edges must be within distance $2$, so there must be an edge between them that is not part of $A$. If $a_1a_{\ell+1} \in E(G)$, then $a_1a_2\ldots a_{\ell+1}$ is a $C_{\ell+1}$. Similarly, if $a_1a_{\ell} \in E(G)$ or $a_2a_{\ell+1} \in E(G)$, then there is a $C_{\ell}$. Finally, if $a_2a_{\ell} \in E(G)$, then there is a $C_{\ell-1}$. So $G$ contains a cycle of order $\ell-1$, $\ell$ or $\ell+1$; contradiction. 

So we may assume that $|W_j^*|=\ell$ and $|W_{j+1}^*|=\ell+2$. To finish the proof, we will derive that $G$ then contains another path of order $\ell+1$, starting and ending on blue edges.

Write $W_{j}^*= w_1 \ldots w_{\ell}$. First, since $|W_{j+1}^*|-|W_j^*|=2$, we must have that $W_{j+1}^*= W_{j}^*w_{\ell+1}w_{\ell+2}$, where $w_{\ell}w_{\ell+1}$ is a red edge and $w_{\ell+1}w_{\ell+2}$ is blue. Second, since $w_1w_2$ and $w_{\ell+1}w_{\ell+2}$ are at distance $2$, there is an edge $e^*$ between them. From this observation, we obtain the desired $P_{\ell+1}$ unless $e^*=w_1w_{\ell+2}$.
Third, $w_1w_2$ and $w_{\ell-1}w_{\ell}$ must be at distance $2$ from each other, so they are connected by an edge $e^{**}$ that is not part of $W_j^*$. This yields a forbidden $C_{\ell}$ or $C_{\ell-1}$, unless $e^{**}=w_2w_{\ell-1}$.
Fourth, note that $w_{\ell-2}w_{\ell-1}$ is red, for otherwise $w_{\ell+1}w_{\ell+2}w_1w_2 \ldots w_{\ell-2}w_{\ell-1}$ would yield the desired $P_{\ell+1}$.

 In summary, we have obtained the cycle $\Gamma= w_{1+1} \ldots w_{\ell+2}$, where $w_1w_2$, $w_{\ell-1}w_{\ell}$ and $w_{\ell+1}w_{\ell+2}$ are blue, and $w_{\ell-2}w_{\ell-1}$ is red. Furthermore, it holds that $w_2w_{\ell-1} \in E(G)$.
 
Next, we are going to focus on the edge $e^{***}=w_{\ell-3}w_{\ell-2}$. Since $\ell \ge 5$, this edge is different from the first edge $w_1w_2$. Suppose that $e^{***}$ is blue. 

Then $w_{\ell-2} w_{\ell-3} \ldots w_{2}w_{\ell-1}w_{\ell}w_{\ell+1}w_{\ell+2}$ forms a $P_{\ell+1}$ starting and ending on blue edges. Suppose on the other hand that $e^{***}$ is red.
Because $e^{***}$ and $w_{\ell-2}w_{\ell-1}$ are consecutive red edges of $W_{j+1}^*$, it follows from the construction of the paths $(W_i^*)_{1\le i \le j+1}$ that there must be a pendant blue edge $w_{\ell-2}w_{p}$ that is only incident to $\Gamma$ in the vertex $w_{\ell-2}$. (This pendant edge used to be the blue end-edge of some preliminary path $W_i^*$, $i<j$.) Now $w_pw_{\ell-2}w_{\ell-3} \ldots w_2w_1 w_{\ell+2}w_{\ell+1}$ forms a $P_{\ell+1}$, starting and ending on blue edges.
\end{proof}

\section{No $4$-cycles}\label{sec:C4}

In this section we prove Theorem~\ref{thm:main,even}\ref{itm:C4}. We proceed by a case analysis. In Subsubcases~\ref{case2.1.2} and~\ref{case2.2.1} and Subsubsubcase~\ref{case2.2.2.2} we can reduce to the case of the neighbourhood of a triangle, which constitutes exactly the extremal hairy triangle. In the other situations, we derive bounds that are smaller, at most $2\Delta$ in particular.

\begin{proof}[Proof of Theorem~\ref{thm:main,even}\ref{itm:C4}]
Let $G$ be a $C_4$-free graph with $\DeltaG=\Delta\ge 4$.
Let $H$ be a \textit{blue} subgraph of $G$ whose edges form a maximum clique in $L(G)^2$. Thus bounding $\omega(L(G)^2)$ is equivalent to bounding $e(H)$. 
Choose an edge $uv\in E(H)$, and define
$A_u= N_G(u) \setminus \{v\}$ and
$A_v= N_G(v) \setminus N_G[u]$. For short we will write $A= A_u \cup A_v$ for the neighbourhood of $\{u,v\}$ and we also need the second-order neighbourhood $B= N_G(A)\setminus (A \cup \{u,v\})$.

Note that the edges of $H$ can be partitioned into the following six sets:
\[
E_H(\{u\},A), \,\,\,\,\,E_H(\{v\},A), \,\,\,\,\,E_H(A_u,B), \,\,\,\,\,E_H(A_v,B), \,\,\,\,\,E(H[A]), \,\,\,\,\,\{uv\}.
\]

We will use the following claim a few times.

\begin{claim}
$e(H[A]) \le 1$.
\end{claim}
 \begin{claimproof} If not, then $G[A]$ must contain a path of order three, which forms a $C_4$ with $u$ and/or $v$.
\end{claimproof}

We now start the case analysis.
\begin{enumerate}[wide, labelwidth=!, labelindent=0pt,noitemsep,font=\bfseries,label=Case \arabic*.,ref=\arabic*]
\item\label{case1} \textit{No vertex in $A$ has two blue neighbours in $B$.}

The first thing to notice is that $A_u$ and $A_v$ each contain at most three vertices with a blue edge to $B$. Indeed, if there are four such vertices $x_1,\ldots,x_4$ with blue neighbours $y_1,\ldots, y_4 \in B$ respectively, then the $(y_i)_{1\le i\le 4}$ must be pairwise distinct to prevent a $C_4$. Therefore the blue edges $(x_iy_i)_{1\le i\le 4}$ are pairwise at distance exactly $2$. There can be at most two edges in $G[\{x_1,x_2,x_3,x_4\}]$ and these must be nonincident, for otherwise they form a $C_4$ with $v$. Say these edges are $x_1x_2$ and $x_3x_4$ (or a subset thereof). Then $y_1y_2y_2y_4$ is a $C_4$, contradiction. 
Second, it cannot be that both $e_H(\{v\},A) \ge 2$ and $e_H(A_u,B)  \ge 2$. Indeed, otherwise there are two vertices in $N_H(v)$ that must be complete to two vertices in $B\cap N_H(A_u)$, thus forming a $C_4$, contradiction. 
So $e_H(\{v\},A) + e_H(A_u,B) \le \Delta$ (where here we also used our assumption that $\Delta \ge 4$). And similarly $e_H(\{u\},A) + e_H(A_v,B) \le \Delta$. It follows that $e(H)$ is at most
\begin{align} \label{eq:firsteq}
&e_H(\{v\},A) + e_H(A_u,B) + e_H(\{u\},A) + e_H(A_v,B) + e(H[A]) + |\{uv\}| \notag \\ 
&\le \Delta+ \Delta + 1 + 1= 2\Delta+2.  
\end{align}
This is bounded from above by $3(\Delta-1)$ if $\Delta \ge 5$. To conclude the same for the case $\Delta=4$, we need to reduce the bound in equation~\eqref{eq:firsteq} by $1$. 

If $e(H[A])=0$, then we get the desired improvement for free.
If, on the other hand, $E(H[A])$ is nonempty, then its unique edge $ab$ has both endpoints in either $N(u)\setminus N(v)$ or in $N(v) \setminus N(u)$, (otherwise $abuv$ would form a $C_4$). Without loss of generality, assume $ab$ is induced by $N(u)\setminus N(v)$. In that case it follows that $E_H(\{v\},A)= \emptyset$ (or otherwise a blue neighbour of $v$ in $A$ would have to be adjacent to $a$ or $b$, forming a $C_4$.)

But then $e_H(\{v\},A) + e_H(A_u,B)\le  0 + \Delta-1$, so we have again reduced the upper bound in~\eqref{eq:firsteq} by $1$, as desired.

\item\label{case2} \textit{At least one vertex in $A$ has two blue edges to $B$.}

Without loss of generality, let $x \in A_u$ be such a vertex and let $x_1^*,x_2^*$ denote two of its blue neighbours in $B$.

\begin{enumerate}[wide, labelwidth=!, labelindent=0pt,noitemsep,font=\bfseries,label=Subcase \arabic{enumi}.\arabic*.,ref=\arabic{enumi}.\arabic*]
\item\label{case2.1}  \textit{$x$ is the only vertex in $A_u$ that has a blue edge to $B$.}
\begin{enumerate}[wide, labelwidth=!, labelindent=0pt,noitemsep,font=\bfseries,label=Subsubcase \arabic{enumi}.\arabic{enumii}.\arabic*.,ref=\arabic{enumi}.\arabic{enumii}.\arabic*]
\item\label{case2.1.1} \textit{$vx \notin E(G)$.}

Suppose there exists $vy \in E_H(\{v\},A_v)$. Then $y\ne x$ because $vx \notin E(G)$. Also, $yx \notin E(G)$ because otherwise $uvyx$ would be a $C_4$. So $y$ must be adjacent to the two blue neighbours $x_1^*, x_2^*$ of $x$ in $B$, in order to have $vy$ within distance $2$ of $xx_1^*$ and $xx_2^*$. But then $xx_1^*x_2^*y$ forms a $C_4$. We deduce 
\begin{align}\label{eq:supidcaseequation1}
E_H(\{v\},A_v) = \emptyset.
\end{align}
We now show that it is impossible for both $e_H(\{u\},A_u) \ge 2$ and $e_H(A_v,B)\ge 1$ to hold. Indeed, suppose there are $x_1,x_2 \in E_H(\{u\},A_u)$ and a blue neighbour $y^*\in B$ of some $y\in A_v$. Since $uy, uy^*, x_1y,x_2y \notin E(G)$ while $yy^*$ must be within distance $2$ of both $ux_2$ and $ux_1$, it follows that $y^*x_1, y^*x_2 \in E(G)$, yielding the $4$-cycle $ux_1x_2y^*$. Contradiction.

If $e_H(A_v,B)=0$ then $e(H) \le |N_H(u) \cup N_H(x)| + e(H[A]) \le (2\Delta-1)+1 \le 3(\Delta-1),$ as desired. So we may from now on assume that 
\begin{align}\label{eq:supidcaseequation2}
e_H(\{u\},A_u)\le 1.
\end{align}
Next, we want to show that $e_H(A_v,B) \le 4$.
Suppose for a contradiction that $e_H(A_v,B) \ge 5$.

Suppose first that there exists $y\in A_v$ with (at least) three blue neighbours $y_1^*,y_2^*,y_3^*$ in $B$. Recall that $x_1^*, x_2^* \in B$ are two blue neighbours of $x$. Since $\{x_1^*,x_2^*\} $ has at most one element in common with $\{y_1^*,y_2^*,y_3^*\}$  (otherwise there is a $C_4$) we may without loss of generality assume that $\{x_1^*,x_2^*\}\cap \{y_1^*,y_2^*\}=\emptyset$. If $xy_1^*,xy_2^*,yx_1^*,yx_2^* \notin E(G)$, then $\{x_1^*,x_2^*\}$ must be complete to $\{y_1^*,y_2^*\}$, yielding a $C_4$. So without loss of generality $xy_1^* \in E(G)$. This implies $xy_2^*,yx_1^*,yx_2^* \notin E(G)$ (otherwise there is a $C_4$ containing $x$ and $y$). So in order to have $yy_2^*$ within distance $2$ of $xx_1^*$ and $xx_2^*$, we must have $x_1^*y_2^*, x_2^*y_2^* \in E(G)$, yielding $xx_1^*y_2^*x_2^*$ as a $C_4$. Contradiction. So we have derived that each $y\in A_v$ has at most two blue neighbours in $B$.

Now suppose that some vertex $y_{12}\in A_v$ has two blue neighbours $y_1^*,y_2^*$ in $B$. By the argument in the previous paragraph, it may not be that $|\{x_1^*,x_2^*\}\cap \{y_1^*,y_2^*\}| \in \{0,2\}$, so without loss of generality $x_2^*=y_2^*$. 

Additionally suppose there is another vertex $y_{34}\in A_v$ with two blue neighbours $y_3^*,y_4^*$ in $B$. By the same argument, one of $\{y_3^*,y_4^*\}$ is equal to one of $\{x_1^*,x_2^*\}$. But $x_2^*=y_2^* \notin \{y_3^*,y_4^*\}$ (otherwise there is a $C_4$ containing $y_2^*$ and $v$), so without loss of generality $y_4^*=x_1^*$.
Since we assumed that $e_H(A_v,B) \ge 5$, there is yet another vertex $y_{5}\in A_v$ with (at least) one neighbour $y_5^*\in B$. Since $y_5y_5^*$ must be within distance $2$ of $xx_1^*$ and $xx_2^*$ it follows that $xy_5^* \in E(G)$. Since $G[A]$ does not contain a path of order $3$ (otherwise there is a $C_4$), at least one of $y_5y_{12}, y_5y_{34}$ is not an edge. Without loss of generality, $y_5y_{12} \notin E(G)$.  Then, in order to have $y_5y_5^*$ within distance $2$ of $y_{12}y_2^*$ and $y_{12}y_1^*$, we must either have $y_5^*y_{12} \in E(G)$ (in which case $y_5^*xy_2^*y_{12}$ is a $C_4$) or $y_5^*y_1^*,y_5^*y_2^* \in E(G)$ (in which case $y_5^*y_1^*y_{12}y_2^*$ is a $C_4$). Contradiction.

Thus we have derived that $y_{12}$ is the only vertex in $A_v$ with two blue neighbours in $B$ (namely $y_1^*$ and $y_2^*$). Since we assumed $e_H(A_v,B) \ge 5$, there are three other vertices $y_3,y_4,y_5 \in A_v$ with unique blue neighbours $y_3^*,y_4^*,y_5^* \in B$, respectively. Since $G[A]$ does not contain a path of order $3$, the complement of the graph induced by $Y=\{y_{12},y_3,y_4,y_5\}$ contains a $C_4$. This implies there is a $C_4$ in the graph induced by $\{y_1^*,y_2^*,y_3^*,y_4^*,y_5^*\}$, the set of blue neighbours of $Y$ in $B$. Contradiction.

Thus, we have derived that no vertex in $A_v$ has more than one blue neighbour in $B$. Now let $y_1,\ldots, y_5 \in A_v$ be vertices with respective unique blue neighbours $y_1^*,\ldots, y_5^* \in B$. Since $G[A]$ does not contain a path of order $3$, the complement of the graph induced by $\{y_{1},y_2,y_3,y_4,y_5\}$ contains a $C_4$. This implies there is a $C_4$ in the graph induced by $\{y_1^*,y_2^*,y_3^*,y_4^*,y_5^*\}$. Contradiction.

This concludes our proof that 
$e_H(A_v,B) \le 4$.
Then together with~\eqref{eq:supidcaseequation1} and~\eqref{eq:supidcaseequation2}, it follows that
\begin{align*}
e(H) &\le e_H(\{u\},A_u) + |\{uv\}| + e_H(\{x\},B) + e_H(A_v,B) + e(H[A])\\
& \le 1 + 1 + (\Delta-1) + 4 = \Delta + 5
\le 3(\Delta-1).
\end{align*}

\item\label{case2.1.2} \textit{$vx \in E(G)$.}

Suppose there exists an edge $yy^* \in E_H(A_v,B)$, with $y \in A_v$. Then the absence of $4$-cycles dictates that $y$ is not adjacent to $x$ nor to any of its blue neighbours in $B$. Therefore $y^*$ is adjacent to all blue neighbours of $x$ in $B$, of which there are at least two by assumption. But then these neighbours form a $C_4$ with $x$ and $y^*$. Contradiction. So $E_H(A_v,B)= \emptyset$ and therefore all edges of $H$ are incident to the triangle $uxv$. So $e(H) = e(G[N_H(u) \cup N_H(x) \cup N_H(v)]) \le 3 (\Delta-1)$.
\end{enumerate}
\item\label{case2.2} \textit{There is another vertex $x_2$ in $A_u$ that has a blue edge to $B$.}

Note that in this case $xx_2 \in E(G)$, for otherwise there would be a $C_4$ in the graph induced by $u,x,x_2$ and the blue neighbours of $x$ and $x_2$ in $B$. Note furthermore that there cannot be a third vertex $x_3\in A_u$ that has a blue edge to $B$, for otherwise the same argument yields $xx_3\in E(G)$ so that $x_2xx_3u$ would yield a $C_4$.

\begin{enumerate}[wide, labelwidth=!, labelindent=0pt,noitemsep,font=\bfseries,label=Subsubcase \arabic{enumi}.\arabic{enumii}.\arabic*.,ref=\arabic{enumi}.\arabic{enumii}.\arabic*]
\item\label{case2.2.1} \textit{$vx \notin E(G)$.}

First, suppose there exists a blue edge $vy \in E_H(\{v\},A)$. Then $y \ne x $ (by assumption) and $y \ne x_2$  and $yx \notin E(G)$ (for otherwise $uvyx$ is a $C_4$). Since $vy$ must be within distance $2$ of the (blue) edges in $E_H(\{x\},B)$, it follows that $y$ must be adjacent to both blue neighbours $x_1^*, x_2^*$ of $x$ in $B$. But then $xx_1^* x_2^* y $ forms a $C_4$. Contradiction. So we conclude that $E_H(\{v\},A)= \emptyset$.
Second, suppose there is an edge $yy^* \in E_H(A_v,B)$, with $y \in A_v$ and $y^*\in B$. Let $z_1^*, z_2^*$ be two blue neighbours of $x$ in $B$ and let $z_3^*$ be a blue neighbour of $x_2$ in $B$.  Recall that $xx_2 \in E(G)$ and, as before, $y \notin \{x,x_2\}$ and $yx, yx_2 \notin E(G)$. So in order to have $yy^*$ within distance $2$ of $xz_1^*, xz_2^*$ and $x_2z_3^*$, we must have for all $i\in \{1,2,3\}$ that either $y^*z_i^* \in E(G)$ or $y^*=z_i^*$, and $y^*$ can be equal to only one of the $z_i^*$. If $y^* =z_3^*$ then $xz_1^*z_2^*y^*$ will form a $C_4$. On the other hand, if (without loss of generality) $y^*=z_1^*$, then $xy*z_3^*x_2$ forms a $C_4$. Contradiction. We conclude that $E_H(A_v,B)$ must be empty too. It follows that all edges of $H$ are incident to the triangle $uxx_2$, so $e(H) \le 3 (\Delta-1)$.

\item\label{case2.2.2} \textit{$vx \in E(G)$.}

By the argument of Subsubcase~\ref{case2.1.2}, $E_H(A_v,B)=\emptyset$.

\begin{enumerate}[wide, labelwidth=!, labelindent=0pt,noitemsep,font=\bfseries,label=Subsubsubcase \arabic{enumi}.\arabic{enumii}.\arabic{enumiii}.\arabic*.,ref=\arabic{enumi}.\arabic{enumii}.\arabic{enumiii}.\arabic*]
\item\label{case2.2.2.1} \textit{$E_H(\{v\}, A_v) \ne \emptyset$.}

Let $vy \in E_H(\{v\},A_v)$ and $x_2x_2^* \in E_H(A_u,B)$. Since $x_2y, vx_2 \notin E(G)$ (otherwise $uvyx_2$  or $ uvx_2 x$ is a $C_4$), we must have $yx_2^* \in E(G)$. This holds for all such pairs, so in order to prevent a $C_4$, we must have $e_H(\{v\},A_v) + e_H(\{x_2\},B) \le 2$. So  $e(H) \le e_H(\{v\},A_v) + e_H(\{x_2\},B) + |N_H(x) \cup N_H(u)| + e_H(A_v,B) \le 2 + (2\Delta-1) + 0 = 2\Delta +1$.
This is bounded from above by $3(\Delta -1)$ if $\Delta \ge 4$, which holds in this subcase because $x$ is adjacent to $u,v,x_2$ and its two or more neighbours in $B$.

\item\label{case2.2.2.2} \textit{$E_H(\{v\}, A_v) = \emptyset$.}

In this case all edges of $H$ are incident to the triangle $uxx_2$, so $e(H)\le 3 (\Delta-1)$. \qedhere

\end{enumerate}
\end{enumerate}
\end{enumerate}
\end{enumerate}
\end{proof}

\section{Bipartite and two forbidden cycle lengths}\label{sec:bip}

In this section, we prove Theorem~\ref{thm:bip}. Let us begin with the basic reduction from the $\{C_3,C_5\}$-free to the bipartite setting.

\begin{lemma}\label{lem:reduction}
Let $\cal G$ be a class of $\{C_3,C_5\}$-free graphs that is invariant under vertex-deletion.
Let $\cal G_{\bip}$ be the class of graphs in $\cal G$ that are bipartite.
Then, provided both are well-defined, $\max_{G\in \cal G} \omega_2'(G) = \max_{G\in \cal G_{\bip}} \omega_2'(G)$.
\end{lemma}

\begin{proof}
Clearly, $\max_{G\in {\cal G_{\bip}}} \omega_2'(G) \le \max_{G\in {\cal G}} \omega_2'(G)$, so it remains to prove the converse.
Given a graph $G$ in ${\cal G}$ we choose a subgraph $H$ of $G$ whose edges form a maximum clique in $L(G)^2$ and we choose an edge $uv\in E(H)$. Consider the induced subgraph $G^*:=G[N(u)\cup N(v)\cup N(N[u]) \cup N(N[v])]$. The fact that $G$ is $\{C_3,C_5\}$-free implies that $G^*\in {\cal G_{\bip}}$. Since $H$ is a subgraph of $G^*$ and moreover all possible edges (in $G$) between edges of $H$ are contained in $G^*$, it follows that $\omega(L(G^*)^2) = e(H)=\omega(L(G)^2)$. We conclude that $\max_{G\in {\cal G}} \omega_2'(G) = \max_{G\in {\cal G}} \omega_2'(G^*) \leq \max_{G\in {\cal G_{\bip}}} \omega_2'(G)$.
\end{proof}

Thus for Theorem~\ref{thm:bip} it suffices to prove the following result.
\begin{theorem}\label{thm:bip,reduced}
For a $\{C_{2k},C_{2k+2}\}$-free bipartite graph $G$ with $\DeltaG=\Delta$, $\omega_2'(G) \le \max\{k\Delta, 2k(k-1)\}$.
\end{theorem}

\begin{proof}
Let $G=G[X\times Y]$ be bipartite and $\{C_{2k},C_{2k+2}\}$-free with $\DeltaG=\Delta$. 
By Theorem~\ref{thm:FSGT}, we may assume throughout that $k\le \Delta$.
Let $H$ be a subgraph of $G$ whose edges form a maximum clique in $L(G)^2$, so that $e(H)= \omega(L(G)^2)$. A path in $G$ will be called \textit{$H$-sided} if it starts and ends on edges of $H$. Given a vertex $v\in V(G)$, an \textit{$H$-neighbour} of $v$ is a vertex $w \in N_H(v)$.

Assume that $\omega(L(G)^2) > \max\{k \Delta, 2k(k-1)\}$. Under this assumption, we want to derive that for any $H$-sided path $P$ of order smaller than $2k+1$, we can find another $H$-sided path that has order $|P|+1$ or $|P|+2$, which is sufficient by the following claim.

\begin{claim}\label{cl:bip_minus2}
Suppose that for each $H$-sided path $P$ in $G$ of order $|P|<2k+1$, we can find another $H$-sided path of order $|P|+1$ or $|P|+2$. Then $G$ contains $P_{2k+1}$ as a subgraph, and also contains a copy of $C_{2k+2}$ or $C_{2k}$.
\end{claim}
\begin{claimproof}
Because $e(H) \ge 1$, there exists an $H$-sided path of order $2$. We can iteratively extend the length of this path by $1$ or $2$, ultimately yielding an $H$-sided path $P$ of order in $\{2k+1,2k+2\}$. In particular, $G$ contains a path of order $2k+1$, as desired. The first and final edge of $P$ are in $H$ and therefore (also using that $|P| \ge 2k+1 \ge 5$) they must be at distance \textit{exactly} $2$. Since $G$ is bipartite, this implies the existence of a cycle of order in $\{ |P|, |P|-2\}$ if $|P|$ is even, and a cycle of order $|P|-1$ if $|P|$ is odd. So $G$ has a cycle of order in $\{2k,2k+2\}$. 
\end{claimproof}

Let $P$ be an $H$-sided path. For clarity of notation we assume from now on that $P$ has even order $2\ell$, for some $\ell \le k$. For paths of odd order $<2k+1$ the arguments are similar and in fact slightly easier, because the bounds we need are slightly more forgiving in that case. Write $P=p_1p_2 \ldots  p_{2\ell}$.

First, we need to introduce some definitions. Let $X_P= X \cap V(P)=p_1p_3\ldots p_{2\ell-1}$ and $Y_P= Y \cap V(P)=p_2p_4 \ldots p_{2\ell}$ be the two parts of the bipartite graph induced by $P$. A vertex of $P$ will be called \textit{$r$-extravert} if its number of $H$-neighbours outside $P$ is at least $r$. For short, we call the vertex \textit{extravert} if it is $1$-extravert. Conversely, a vertex of $P$ is \textit{introvert} if all of its $H$-neighbours are in $P$. By $P_{\Ext}^{(r)}$ and $P_{\Ext}$ we denote the set of $r$-extravert vertices and extravert vertices respectively, and $P_{\Int}$ denotes the set of introvert vertices. Finally, by $\Obs(P)$  we will denote the set of \textit{obsolete edges}, which by definition are those edges of $H$ that are incident to some vertex of $P \setminus \{p_1,p_{2\ell} \}$. We call them obsolete because they cannot be `greedily' used to extend the order of $P$.

From now on, suppose for a contradiction that it is \textit{not} possible to find an $H$-sided path of order $|P|+1$ or $|P|+2$. Then the following claims hold.

\begin{claim}\label{cl:bip0} The first and final vertex of $P$ are introvert.
\end{claim}
\begin{claimproof}
Suppose by symmetry that the first vertex $p_1$ is extravert. Then it has an $H$-neighbour $p_0$ outside $P$, so $p_0P$ is an $H$-sided path of order $|P|+1$. Contradiction.
\end{claimproof}

\begin{claim}\label{cl:bip_minus1}
 $|\Obs(P)| >  \max\{k\Delta, 2k(k-1)\}$.
\end{claim}
\begin{claimproof}
Suppose not. Then $|\Obs(P) | \le \max\{k\Delta, 2k(k-1) \} < e(H)$. Therefore there exists an edge $e^*$ in $H$ that is not incident to any vertex of $P$. The final edge $e$ of $P$ is in $H$, so $e^*$ and $e$ must be at distance \textit{exactly} $2$. This implies that we can extend $P$ to an $H$-sided path (ending on $e^*$ rather than $e$) that is of order $|P|+1$ or $|P|+2$. Contradiction.
\end{claimproof}

So in order to arrive at a contradiction, it suffices to show that either $|\Obs(P)| \le k \Delta$ or $|\Obs(P)| \le  2k(k-1)$. We will now derive some structural properties of our counterexample.

\begin{claim}\label{cl:bip1}
 Any two extravert vertices in the same part (both in $X_P$ or both in $Y_P$) have a common neighbour outside $P$.
\end{claim}
\begin{claimproof}
Indeed, suppose without loss of generality that $p_i,p_j$ are two extravert vertices in $X_P$, with $H$-neighbours $q_i$ respectively $q_j$ outside $P$. If $q_i=q_j$ we are done, so suppose $q_i\ne q_j$. The edges $p_iq_i$ and $p_jq_j$ need to be within distance $2$. Since odd cycles are not allowed in $G$, it follows that $p_ip_j, q_iq_j \notin E(G)$, so $q_i$ or $q_j$ must be a common neighbour of $p_i$ and $p_j$.
\end{claimproof}

\begin{claim}\label{cl:bip2}
$P$ contains at most two pairs of consecutive extravert vertices, and if there are two such pairs $p_i p_{i+1}$ and $p_j p_{j+1}$, then they must have different parity, in the sense that $i=j+1$ (mod $2$).
\end{claim}
\begin{claimproof}
Suppose there are two extravert pairs $p_ip_{i+1}$, $p_jp_{j+1}$ of the same parity. Then without loss of generality $i+1 < j$ and $p_i,p_j \in X_P$ and $p_{i+1},p_{j+1} \in Y_P$. See Figure~\ref{fig:claimbip2and7}. By Claim~\ref{cl:bip1}, $p_i$ and $p_j$ have a common neighbour $u \in Y \setminus Y_P$, and $p_{i+1}$ and $p_{j+1}$ have a common neighbour $v \in X \setminus X_P$. Therefore we can replace the subpath $P^*=p_i p_{i+1}\ldots p_j p_{j+1}$  of $P$ by $ p_i u p_j p_{j-1} \ldots p_{i+2} p_{i+1} v p_{j+1}$, which uses the same vertices as $P^*$ \textit{and} two extra vertices $u,v$ outside of $P$.  Thus, we have constructed an $H$-sided path of order $|P|+2$. Contradiction.
\end{claimproof}

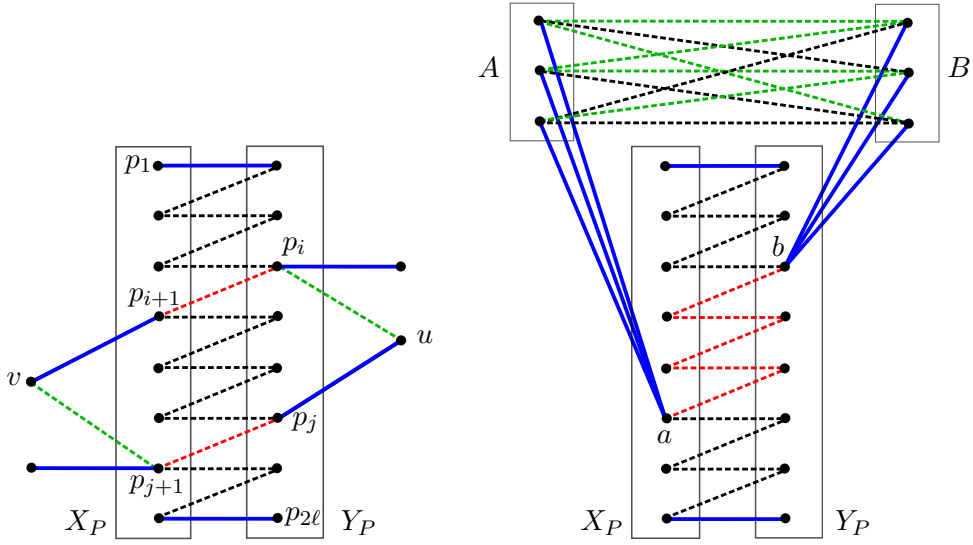
\begin{figure}
 \begin{center}
   \begin{overpic}[width=0.95\textwidth]{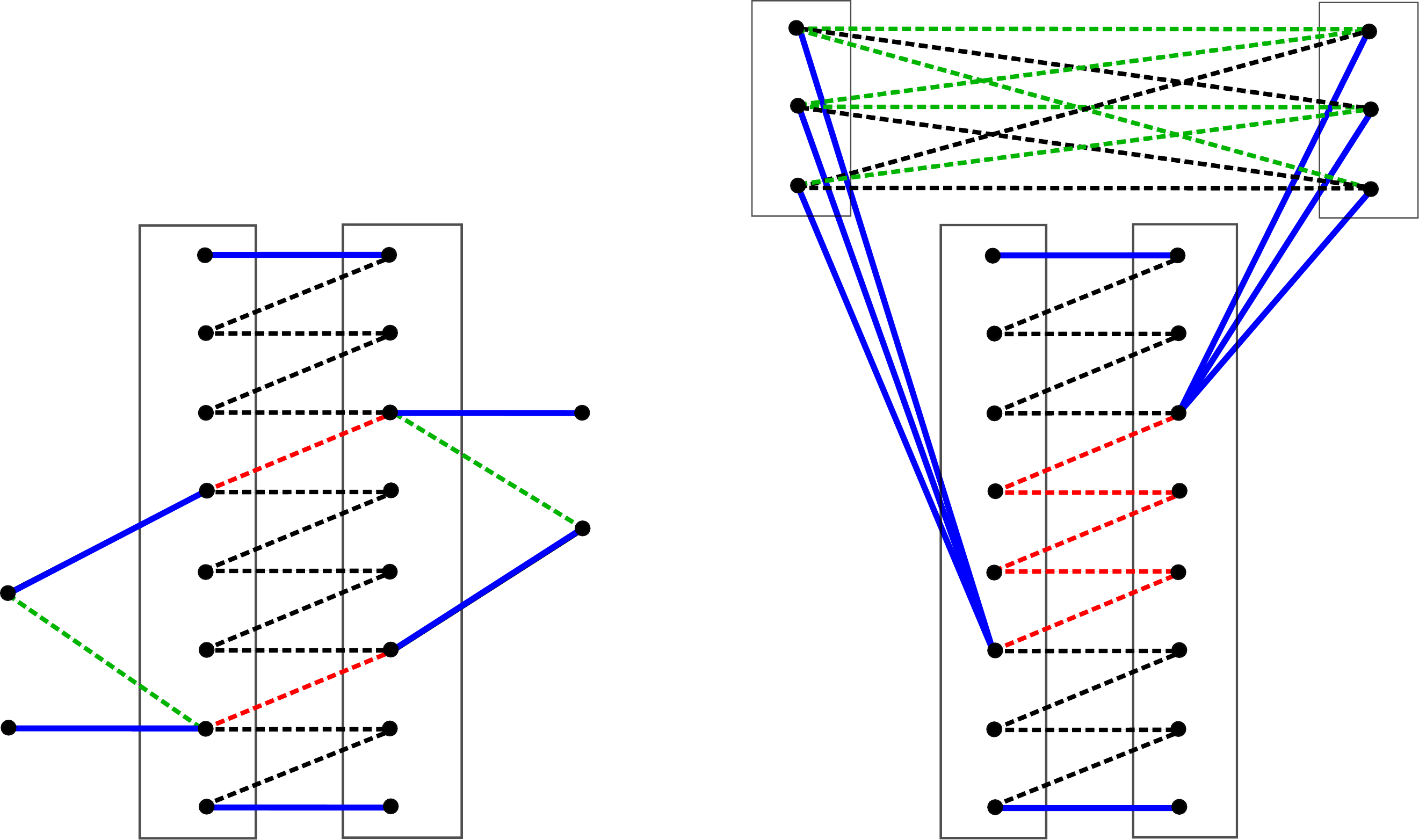}
    \put (9.8,41){ $p_1$}
    \put (27,31.9){ $p_{i}$}
    \put (10.2,26.3){ $p_{i+1}$}
    \put (28.1,12.6){ $p_{j}$}
    \put (10.3,5.6){ $p_{j+1}$}
    \put (27.3,2.1){ $p_{2\ell}$}
    \put (-3.2,17){ $v$}
    \put (41.7,21.6){ $u$}
    
    \put (48.4,51){ $A$}
    \put (99.8,51){ $B$}
    \put (68,10.65){ $a$}
    \put (80.6,31.2){ $b$}   
     
    \put (3.15,1){ $X_P$} 
    \put (33.4,1){ $Y_P$} 
    \put (59.65,1){ $X_P$} 
    \put (87.5,1){ $Y_P$} 
 \end{overpic}
  \caption{A depiction of the contradictory path-extensions described by Claim~\ref{cl:bip2} (left) and Claim~\ref{cl:bip7} (right). On the right, $a$ and $b$ are non-adjacent $3$-extravert vertices and the subpath of $P$ between $a$ and $b$ has order $6$. This means that $a$ and $b$ are too close to each other, with respect to $P$. Indeed, by following the green edges (and two blue edges) rather than the red edges, we obtain an $H$-sided path of order $|P|+2$. \label{fig:claimbip2and7}}
\end{center}
\end{figure}

The next claim is arguably the heart of the argument.
\begin{claim} \label{cl:bip3}
There are at most $\ell$ extravert vertices.
\end{claim}
\begin{claimproof}

Consider the vertex pairs $(p_2,p_3)$, $(p_4,p_5)$, $\ldots$, $(p_{2\ell-2},p_{2\ell-1})$. By Claim~\ref{cl:bip0}, all extravert vertices are contained in the union of these $\ell-1$ pairs. So if there are more than $\ell$ extravert vertices, then by the pigeonhole principle at least two pairs entirely consist of extravert vertices. We have obtained two same-parity pairs of consecutive extravert vertices, contradicting Claim~\ref{cl:bip2}.
\end{claimproof}

 From now on, let $r \ge 0$ be the maximal integer (if it exists) 
  such that there are nonadjacent $r$-extravert vertices $s,t$ with $s \in X_P$ and $t\in Y_P$. 
\begin{claim} \label{cl:bip4} 
The integer $r$ is well-defined.
\end{claim}
\begin{claimproof}
Suppose $r$ does not exist. Then the vertices of $P$ induce a complete bipartite graph, with parts $X_P$ and $Y_P$. By Claim~\ref{cl:bip3} we have $|P_{\Ext}| \le \ell$, and therefore $|\Obs(P)| \le |X_P| |Y_P| + |P_{\Ext}| (\Delta- \min\{|X_P|,|Y_P|\}) = \ell^2 + \ell (\Delta-\ell) \le k \Delta$, contradicting Claim~\ref{cl:bip_minus1}.
\end{claimproof}

The next claim follows directly from the definition of $r$.

\begin{claim} \label{cl:bip5} 
 The graph induced by $P_{\Ext}^{(r+1)}$ is complete bipartite.
\end{claim}

Next, we show that highly extravert vertices of different parity cannot be too close to eachother with respect to $P$.

\begin{claim} \label{cl:bip7}
Let $q$ be a positive integer. Let $a\in X_P, b \in Y_P$ be two non-adjacent $q$-extravert vertices.
Then the subpath of $P$ having endpoints $a$ and $b$ has at least $2q+2$ vertices.
\end{claim}
\begin{claimproof}
Suppose for a contradiction that the subpath of $P$ with endpoints $a$ and $b$ has (even) order $d\leq 2q$.
Let $A= \{a_1, \ldots, a_q \}$ denote a subset of the $H$-neighbours of $a$ in $Y \setminus Y_P$. Similarly, let $B=\{b_1, \ldots, b_q \}$ denote a subset of the $H$-neighbours of $b$ in $X \setminus X_P$. See Figure~\ref{fig:claimbip2and7}. Because $ab \notin E(G)$ and the $H$-edges $a_ia$, $b_jb$ should be within distance $2$ for all $i$, $j$, it follows that $A$ is complete to $B$. Therefore there exists a path $P^*=aa_1b_1a_2b_2 \ldots a_q b_q b$ of order $d+2$ that only intersects $P$ in $a$ and $b$. This leads to a contradiction, because it implies that we can construct an $H$-sided path of order $|P|+2$, by replacing the order $d$ subpath of $P$ between $a$ and $b$ with the order $d+2$ path $P^*$.
\end{claimproof}

With the above claims, we now complete the proof of Theorem~\ref{thm:bip,reduced} by deriving a contradiction to Claim~\ref{cl:bip_minus1}.

We partition the vertices of $P$ and estimate the $H$-edges incident to them separately. First we need some definitions. Let $i_x=|P_{\Ext}^{(r+1)} \cap X_P|$ and $i_y=|P_{\Ext}^{(r+1)} \cap Y_P|$ be the numbers of $(r+1)$-extravert vertices in the parts $X_P, Y_P$ of the bipartite graph induced by $P$. Similarly, let $j_x= |P_{\Ext}\setminus P_{\Ext}^{(r+1)} \cap X_P|$ and $j_y=|P_{\Ext}\setminus P_{\Ext}^{(r+1)} \cap Y_P|$ be the number of vertices that are extravert but not $(r+1)$-extravert, in part $X_P$ respectively $Y_P$. Note that the remaining $|X_P|-i_x-j_x$  (resp. $|Y_P|-i_y-j_y$)  vertices in $X_P$ (resp. $Y_P$) are introvert. 

An important observation\label{edgepartition:bip9} is that we can write $\Obs(P)$ as a disjoint union $E_1 \cup E_2 \cup E_3$, where 
\begin{align*}
E_1 = I_H(P_{\Ext}^{(r+1)}), \,\,\,\,\, 
E_2 = I_H(P_{\Ext})\setminus I_H(P_{\Ext}^{(r+1)}), \,\text{ and }\, 
E_3 = E(H[P_{\Int}]). 
\end{align*}

Recall from Claim~\ref{cl:bip5} that $G[P_{\Ext}^{(r+1)}]$ is complete bipartite, so it is efficient to estimate $|E_1|$ by summing the degrees (with respect to $G$) of $P_{\Ext}^{(r+1)}$ and subtracting the double-counted edges of $G[P_{\Ext}^{(r+1)}]$. This yields
 \begin{align}\label{eq:estimateE1}
 |E_1|\le - e(G[P_{\Ext}^{(r+1)}]) + \sum_{v \in P_{\Ext}^{(r+1)}} | N_G(v) |  \le  -i_x i_y + (i_x+i_y) \Delta.
 \end{align}
To estimate $|E_2|$, note that it is maximised if each vertex $v \in P_{\Ext} \setminus P_{\Ext}^{(r+1)}$ has exactly $r$ $H$-neighbours outside $G[P]$ and is incident to all vertices of the opposite part that are not in $P_{\Ext}^{(r+1)}$ (and leaving out one single edge from this graph, to comply with the non-edge that defines $r$). In this case
\begin{align}\label{eq:estimateE2}
|E_2| &\le - e(H[P_{\Ext} \setminus P_{\Ext}^{(r+1)}]) + \sum_{v \in P_{\Ext} \setminus P_{\Ext}^{(r+1)}} |N_H(v) \setminus P_{\Ext}^{(r+1)}|\\ \nonumber
 & \le -j_xj_y + j_x (r+|Y_P|-i_y) + j_y (r+|X_P|-i_x).
\end{align}
The quantity $|E_3|$ is maximised if $P_{\Int} $ induces a complete bipartite graph, so 
 \begin{align}\label{eq:estimateE3}
 |E_3| \le (|X_P|-i_x-j_x) (|Y_P|-i_y-j_y).
 \end{align}

Summing estimates~\eqref{eq:estimateE1},~\eqref{eq:estimateE2} and~\eqref{eq:estimateE3}, we conclude that 
\begin{align}\label{eq:estimateE1E2E3}
|\Obs(P)| 
&\le (i_x+i_y) \Delta + (j_x+j_y) r + |X_P| |Y_P| - i_x |Y_P| -i_y |X_P| \nonumber \\
&= (i_x+i_y) (\Delta-\ell) + (j_x+j_y) r + \ell^2 .
\end{align}
If $\Delta-\ell \ge r$ then (\ref{eq:estimateE1E2E3}) is maximised for $j_x+j_y=0$, so that $i_x+i_y=|P_{\Ext}|$. This means that all extravert vertices are in fact $(r+1)$-extravert. By Claim~\ref{cl:bip3},  
\[|\Obs(P)|\le |P_{\Ext}| (\Delta-\ell) +\ell^2 \le \ell (\Delta-\ell)+\ell^2 \le k \Delta,\]
a contradiction to Claim~\ref{cl:bip_minus1}.
Conversely, if $\Delta-\ell < r$ then 
 the upper bound  on $|\Obs(P)|$ is maximised for $i_x+i_y=0$, so that $j_x+j_y=|P_{\Ext}|$. This means that none of the extravert vertices is $(r+1)$-extravert. By Claim~\ref{cl:bip3}, we again obtain a contradiction to Claim~\ref{cl:bip_minus1}:
\begin{align*}
|\Obs(P)|\le |P_{\Ext}| r +\ell^2 \le \ell (\ell-2) +\ell^2 \le 2 k(k-1).
\end{align*}
In the last line, we used that $r \le \ell-2$, which follows from Claim~\ref{cl:bip7} and the fact that the first and final vertex of $P$ are introvert.
\end{proof}

\bibliographystyle{abbrv}
\bibliography{strongclique}

\end{document}